\documentclass[11pt]{article}

\usepackage{amsfonts,amsmath,amssymb,cite}
\usepackage{amsmath,amsthm}
\usepackage{enumerate}
\usepackage{latexsym}
\usepackage{color}
\usepackage[dvipsnames]{xcolor}
\usepackage{tikz}
\usetikzlibrary{decorations.pathmorphing}
\usetikzlibrary{shapes.geometric}
\usetikzlibrary{svg.path}

\newtheorem{thm}{Theorem}

\newtheorem{ob}{Observation}
\newtheorem{lem}[thm]{Lemma}

\newtheorem{cor}[thm]{Corollary}
\newtheorem{prop}[thm]{Proposition}

\newtheorem{prob}{Problem}

\theoremstyle{definition}

\theoremstyle{remark}

\newcommand{\Kirsti}[1]{\textcolor{red}{#1}}
\newcommand{\DR}[1]{\textcolor{blue}{#1}}

\newcommand{\ggr}{\gamma_{\rm gr}}

\newcommand{\cp}{\,\square\,}

\newcommand{\cF}{{\cal F}}
\newcommand{\cFa}{{\cal F_{\alpha}}}

\newcommand{\vertex}{\node[vertex]}
\tikzstyle{vertex}=[circle, draw, inner sep=0pt, minimum size=6pt]

\begin{document}

\title{Graphs with equal Grundy domination and independence number}
\author{
G\'{a}bor Bacs\'{o}$^{a}$ \and Bo\v{s}tjan Bre\v{s}ar$^{b,c}$ \and
Kirsti Kuenzel$^{d}$ \and Douglas F.~Rall $^{e}$}

\maketitle

\begin{center}
$^a$ Institute for Computer Science and Control, ELKH, Budapest, Hungary \\
$^b$ Faculty of Natural Sciences and Mathematics, University of Maribor, Slovenia\\
$^c$ Institute of Mathematics, Physics and Mechanics, Ljubljana, Slovenia \\
$^d$ Department of Mathematics, Trinity College, Hartford, CT, USA\\
$^e$ Professor Emeritus of Mathematics, Furman University, Greenville, SC, USA\\
\end{center}
\medskip

\maketitle
\begin{abstract}
The Grundy domination number, ${\gamma_{\rm gr}}(G)$, of a graph $G$ is the maximum length of a sequence $(v_1,v_2,\ldots, v_k)$ of vertices in $G$ such that for every $i\in \{2,\ldots, k\}$, the closed neighborhood $N[v_i]$ contains a vertex that does not belong to any closed neighborhood $N[v_j]$, where $j<i$. It is well known that the Grundy domination number of any graph $G$ is greater than or equal to the upper domination number $\Gamma(G)$, which is in turn greater than or equal to the independence number $\alpha(G)$. In this paper, we initiate the study of the class of graphs $G$ with $\Gamma(G)={\gamma_{\rm gr}}(G)$ and its subclass consisting of graphs $G$ with $\alpha(G)={\gamma_{\rm gr}}(G)$. We characterize the latter class of graphs among all twin-free connected graphs, provide a number of properties of these graphs, and prove that the hypercubes are members of this class.  In addition, we give several necessary conditions for graphs $G$ with $\Gamma(G)={\gamma_{\rm gr}}(G)$ and present large families of such graphs.

\end{abstract}

{\small \textbf{Keywords:} } Grundy domination, independence number, upper domination number, bipartite graph\\
\indent {\small \textbf{AMS subject classification:} 05C69, 05C75 }
\section{Introduction} \label{sec:intro}

Given a graph $G$, a set $D$ is a {\em dominating set} if every vertex in $V(G)-D$ has a neighbor in $D$.  The {\em domination number} of $G$ is defined as $\gamma(G)=\min\{|D|:\, D \textrm{ is a dominating set of } G\}$.
A vertex $x$ {\em dominates} a vertex $y$ if $y$ is a neighbor of $x$ or $y=x$.  Building a dominating set in $G$ can be viewed as a process of adding vertices from $G$ to $D$ one by one so that each time a vertex $x$ is added to $D$ it dominates a vertex that was not dominated by vertices added to $D$ before $x$. The size of a largest dominating set obtained by such a process is the {\em Grundy domination number}, $\ggr(G)$, of $G$. Grundy domination was introduced in~\cite{br-go-mi-ra-ri-2014} 
and studied by a number of authors, see~\cite{bell-2021, br-bu-go-kl-ko-pa-tu-vi-2016, br-bu-go-kl-ko-pa-tu-vi-2017, br-ko-to-2019, camp-2021, li-2019, na-to-2020} for a selection of papers on this parameter.

It follows from the definitions that $\gamma(G)\le \ggr(G)$ in any graph $G$, and often the Grundy domination number is much larger than the domination number of $G$. In the seminal paper from 2014~\cite{br-go-mi-ra-ri-2014}, the question of which graphs $G$ enjoy $\gamma(G)=\ggr(G)$ was considered. It was proved that $\gamma(G)=\ggr(G)=1$ only in complete graphs and $\gamma(G)=\ggr(G)=2$ precisely in graphs $G$ whose complement $\overline{G}$ is the disjoint union of one or more complete bipartite graphs. A few years later, Erey proved that the mentioned classes of graphs are the only connected graphs in which equality $\gamma(G)=\ggr(G)=2$ holds~\cite{er-2019}. We mention that an analogous question for two related parameters, the total domination number and the Grundy total domination number, was intensively studied~\cite{bah-2021, br-he-ra-2016, dravec-2022}, yet a complete characterization seems to be elusive.

Since the complete characterization of the graphs $G$ with $\gamma(G)=\ggr(G)$ has been found, natural questions appear by involving graph parameters that lie between $\gamma$ and $\ggr$. Two such important parameters (namely, the independence number $\alpha(G)$ and the upper domination number $\Gamma(G)$) will be considered in this paper, and graphs $G$ in which $\ggr(G)$ is equal to one of these parameters will be studied.
In the next two subsections,  we (1) give some necessary definitions and present basic observations that arise, and (2) formulate the main results of the paper and its organization.

\subsection{Definitions and preliminaries}

Let $G$ be a finite, simple graph with vertex set $V(G)$ and edge set $E(G)$.  (When there is no chance of confusion we will
shorten this notation by setting $V=V(G)$ and $E=E(G)$.)  The order of $G$ will be denoted by $n(G)$.   For a vertex $x \in V$, the {\em open neighborhood of $x$}
is the set $N(x)$ defined by $N(x)=\{w \in V:\, xw \in E\}$.  The {\em closed neighborhood} $N[x]$ is
$N(x) \cup \{x\}$.   The open neighborhood of a set $A \subseteq V$ is $N(A)=\cup_{a\in A}N(a)$ and its
closed neighborhood is $N[A]=N(A) \cup A$.  Two vertices $u$ and $v$ of $G$ are {\em twins} if $N[u]=N[v]$, and we say $G$ is {\em twin-free}
if it has no twins.  For $a \in A$, the {\em private neighborhood} of $a$ (with respect to $A$) is denoted by pn$[a,A]$ and is defined
by pn$[a,A]=\{w \in V:\, N[w] \cap A=\{a\}\}$.  Any vertex in pn$[a,A]$ is called a {\em private neighbor} of $a$ with respect to $A$.
The subgraph of $G$ induced by $A$ is denoted by $G\langle A\rangle$, and for a positive integer $n$, we will use $[n]$ to
denote the set of positive integers not larger than $n$.

A set $A$ of vertices  is a {\em dominating} set of $G$ if $N[A]=V$.  A dominating set $A$ is a minimal dominating set if $A-\{a\}$ does not dominate
$G$ for every $a \in A$.  (Equivalently, pn$[a,A]\neq \emptyset$ for every $a \in A$.)  Imposing this minimality condition while not requiring the set
to be dominating leads to the concept of irredundance.  The set $A$ is {\em irredundant} in $G$ if pn$[a,A]\neq \emptyset$ for every $a \in A$.
Note that a dominating set of $G$ is a minimal dominating set only if it is a maximal irredundant set in $G$.  The {\em domination number} of $G$, denoted $\gamma(G)$,
is the minimum cardinality of a dominating set of $G$.  The {\em upper domination number} of $G$ is the cardinality of a largest minimal dominating set
of $G$ and is denoted $\Gamma(G)$.  The minimum cardinality of a maximal irredundant set in $G$ is the {\em irredundance number} of $G$ and is denoted by
$ir(G)$ while the {\em upper irredundance number}, $IR(G)$, is the maximum cardinality of an irredundant set in $G$.  The {\em independence number} of $G$
is denoted $\alpha(G)$ and is the maximum cardinality of a subset of vertices in $G$ that are pairwise nonadjacent; $i(G)$ denotes the {\em independent domination number}
of $G$, which is the minimum cardinality of a dominating set that is also independent.  Equivalently, $i(G)$ is the minimum cardinality of a maximal
independent set.  If $A$ is a minimal dominating set of cardinality $\gamma(G)$ (respectively, $\Gamma(G)$), then $A$ will be called a $\gamma(G)$-set (respectively, a $\Gamma(G)$-set).  Similar language will be used for each of these other four graphical invariants
$ir,i,\alpha$ and $IR$.  If $I$ is an independent set in $G$, then $x \in {\rm pn}[x,I]$ for every
vertex $x$ in $I$, which implies that $I$ is irredundant.  For any graph $G$, the following well known and much studied string of inequalities

\begin{equation} \label{eqn:commondomstring}
ir(G) \le \gamma(G) \le i(G) \le \alpha(G) \le \Gamma(G) \le IR(G)
\end{equation}
follows from these definitions.

In what follows we will need the following result of Cockayne et al.
\begin{thm} {\rm (\cite[Theorem 5]{co-fa-pa-th-1981})} \label{thm:topthreebipartite}
If $G$ is a bipartite graph, then $\alpha(G)=\Gamma(G)= IR(G)$.
\end{thm}

Let $S=(x_1,\ldots,x_n)$ be a sequence of distinct vertices in $G$.  We denote the length of $S$ by $|S|$.  The set $\{x_1,\ldots,x_n\}$ whose elements are the vertices in $S$ is denoted  by $\widehat{S}$. The sequence $S$ is called a  {\em closed neighborhood sequence} (or a {\em legal sequence}) if
\begin{equation}
\label{eq:defGrundy}
N[x_{i+1}]-\bigcup_{j=1}^{i}N[x_j]\not=\emptyset
\end{equation}
for each $i\in [n-1]$. That is, $(x_1,\ldots,x_n)$ is a closed neighborhood sequence if $x_{i+1}$ has a private neighbor with respect to $\{x_1,\ldots,x_{i+1}\}$ for each $i\in [n-1]$.  We will also say that $x_{i+1}$ {\em footprints} the vertices from $N[x_{i+1}] - \bigcup_{j=1}^{i}N[x_j]$ with respect to $S$, and that $x_{i+1}$ is the {\em footprinter} of any $v\in N[x_{i+1}] - \bigcup_{j=1}^{i}N[x_j]$.   If $S$ is a legal sequence and $\widehat{S}$ is a dominating set of $G$, then $S$ is called a {\em dominating sequence} in $G$.  It is clear that for a dominating sequence $S$ each vertex in $V$ has a unique footprinter in $\widehat{S}$.  Hence,  the function $f_S:V\to \widehat{S}$ that maps each vertex to its footprinter is well defined.
Clearly, a shortest possible dominating sequence has length $\gamma(G)$. A longest possible dominating sequence in $G$ is called a {\em Grundy dominating sequence},  and its length is the {\em Grundy domination number} of $G$, denoted $\ggr(G)$.  Legal sequences were introduced in~\cite{br-go-mi-ra-ri-2014} as sequences of legal moves in the domination game played on a graph. If Staller is the only player making moves, then the length of the resulting dominating sequence  is the Grundy domination number of the graph. See the book~\cite{book-domination-game} for more on domination game and its relations with dominating sequences. 

Note that any legal sequence in $G$ can always be extended (if it is not already) to a dominating sequence in $G$.  Thus, any longest legal sequence in $G$ is a Grundy dominating sequence.  A legal sequence that remains legal under any permutation of its vertices is said to be {\em commutative}.

The next observation follows immediately from the definitions.

\begin{ob} \label{ob:commutative}
If $S=(x_1,\ldots,x_k)$ is a legal sequence in $G$ such that $\widehat{S}$ is irredundant, then $S$ is commutative.
\end{ob}

If $A$ is an irredundant set in $G$, then any permutation of the vertices in $A$ forms a legal sequence.    If $A$ is also a dominating set, then
this sequence is a dominating sequence.  On the other hand, if $A$ does not dominate $G$, then, as noted above, the sequence can be extended to a dominating sequence in $G$.  This immediately implies that $IR(G) \le \ggr(G)$, and thus for any $G$, we have the following extension of~\eqref{eqn:commondomstring}
\begin{equation} \label{eqn:expandedstring}
ir(G) \le \gamma(G) \le i(G) \le \alpha(G) \le \Gamma(G) \le IR(G) \le \ggr(G).
\end{equation}

The join $G\oplus H$ of graphs $G$ and $H$ is the graph obtained from the disjoint union of $G$ and $H$ by adding the edges from the set $\{gh:\, g\in V(G) \text{ and } h\in V(H)\}$.
Given two graphs $G$ and $H$, the {\em Cartesian product} $G\cp H$ of $G$ and $H$ is the graph with $V(G\cp H)=V(G)\times V(H)$ and $(g,h)(g',h')\in E(G\cp H)$ whenever ($g=g'$ and $hh'\in E(H)$) or ($gg'\in E(G)$ and $h=h'$). Cartesian product is associative and commutative. The {\em $k$-cube}, $Q_k$, or the {\em hypercube of dimension $k$}, is the Cartesian product of $k$ copies of the graph $K_2$.

\subsection{Goal and brief outline of the paper}

In this paper, we initiate the study of two natural classes of graphs that arise from involving three invariants in the above inequality chain~\eqref{eqn:expandedstring}.   Notably, we consider the graphs $G$ with $\Gamma(G)=\ggr(G)$, and the graphs $G$ with $\alpha(G)=\ggr(G)$.

It is easy to see that if $u$ is a twin in the graph $G$, then $\Gamma(G)= \Gamma(G-u)$
and $\ggr(G)=\ggr(G-u)$.  Therefore, we will assume that all the graphs under investigation are twin-free.  Furthermore, $\Gamma(G)=\ggr(G)$ if and only if $\Gamma(C)=\ggr(C)$, for each component $C$ of $G$.  With this in mind we let $\cF$ denote the class of twin-free, connected graphs $G$ for which $\Gamma(G)=\ggr(G)$.  In addition we
let $\cFa$ be the subclass of $\cF$ consisting of those $G \in \cF$ such that $\alpha(G)=\ggr(G)$. Note that $\cFa$ is a proper subclass of $\cF$ as can be seen by
$\Gamma(K_n\cp K_2)=n=\ggr(K_n\cp K_2)$ and $\alpha(K_n \cp K_2)=2$.

Graphs in $\cF$ and $\cFa$ are in some sense very special, since $\gamma_{gr}(G) - \Gamma(G)$ and $\gamma_{gr}(G) - \alpha(G)$ can be arbitrarily large.  Indeed, consider the following family of graphs.
For each positive integer $n$ let $V(G_n)=\{x_1, \ldots, x_n,y_1,\ldots, y_n,z_1, \ldots, z_n\}$.  Both of the sets $\{y_1, \ldots, y_n\}$ and
$\{z_1, \ldots, z_n\}$ induce a complete subgraph in $G_n$ and $\{x_1, \ldots, x_n\}$ is independent.  The remaining edges of $G_n$
are $x_iy_i$ and $y_iz_i$ for each $i \in [n]$.
Now, $\gamma_{gr}(G_n) = 2n$, since $S=(x_1, x_2, \ldots, x_n,y_1,y_2, \ldots, y_n)$ is a Grundy dominating set. On the other hand, $\Gamma(G_n)=n+1$ and $\alpha(G_n)=n+1$.

In Section~\ref{sec:F} we consider the class $\cF$ of graphs $G$ with $\Gamma(G)=\ggr(G)$. Note that any minimal dominating set $D$ of size $\Gamma(G)$ gives rise to the partition of  $G$ into vertices in $D$, the private neighbor sets for all vertices in $D$, and the remaining vertices (which are not in $D$ and have at least two neighbors in $D$). We present several necessary conditions that a graph in $\cF$ must possess, which are expressed in terms of the private neighborhoods of vertices in a minimal dominating set of size $\Gamma(G)$.  While these conditions do not necessarily give rise to a characterization of graphs in $\cF$, in Section~\ref{sec:Examples} we present several families of graphs that belong to $\cF$ or even to $\cFa$. We prove that the operation of join preserves the property of a graph being in $\cF $, and provide necessary conditions on graphs $G$ and $H$ whose Cartesian product $G\cp H$ belongs to $\cFa$. In addition, we use some connections with linear algebra to prove that all hypercubes belong to $\cFa$. Section~\ref{sec:cFa} is about graphs in $\cFa$ and is the most extensive one. We prove several necessary and sufficient conditions that a triangle-free graph in $\cFa$ must possess. Most of these conditions are of structural nature and can be expressed as properties that are related to a maximum independent set of a graph. In particular, they lead to a characterization of bipartite graphs in $\cFa$ whose girth is at least $6$. Finally, in Theorem~\ref{thm:main} we give our main result, which is a characterization of graphs in $\cFa$ among all graphs. The characterization is not structural, since it relies on specific properties that any legal closed neighborhood sequence must possess. Nevertheless, it implies a characterization of $n$-crossed prisms that belong to $\cFa$. We conclude the paper with several remarks and open problems.

\section{The Class $\cF$}
\label{sec:F}
In this section, we derive a number of properties that hold for any $\Gamma(G)$-set if $G \in \cF$.  First we see that the class $\cF$
coincides with the class of twin-free, connected graphs whose upper irredundance number equals its Grundy domination number.

\begin{lem} \label{lem:GrundyEqualsIR}
If $G$ is any twin-free, connected graph and $IR(G)=\ggr(G)$, then $G \in \cF$.
\end{lem}
\begin{proof}
Suppose $IR(G)=n$ and let $S=(x_1,\ldots,x_n)$ be a legal sequence formed from a $IR(G)$-set $\widehat{S}$.  Since $|S|=IR(G)=\ggr(G)$ by assumption,
it follows that $\widehat{S}$ is a dominating set (for otherwise $S$ could be extended to a dominating sequence).  Also, $\widehat{S}$ is a minimal dominating set since $\widehat{S}$ is irredundant.  Now we get
\[\Gamma(G)\ge |\widehat{S}|=\ggr(G)=IR(G)\ge \Gamma(G)\,,\]
and thus $\Gamma(G)=\ggr(G)$.
\end{proof}

\begin{lem} \label{lem:pn-a-clique}
Let $G \in \cF$.  If $D$ is any $\Gamma(G)$-set, then {\rm pn}$[u,D]$ induces  a complete subgraph of $G$ for every $u \in D$.

\end{lem}
\begin{proof}
Let $D=\{x_1,\ldots,x_n\}$ be a minimal dominating set of cardinality $\Gamma(G)$ and suppose that $v_1$ and $v_2$ are distinct vertices in
pn$[x_1,D]$.  The sequence $(x_1,x_2,\ldots,x_n)$ is a Grundy dominating sequence.  If $v_1v_2 \notin E$, then $(x_2,x_3,\ldots,x_n,v_1,v_2)$ is
a legal sequence of length $n+1$ since $v_1$ footprints itself and $v_2$ footprints itself.  This  contradiction implies that $G\langle{\rm pn}[x_1,D]\rangle$ is
a complete subgraph.  Since any permutation of $D$ is a legal sequence, the lemma follows.
\end{proof}

\begin{cor} \label{cor:trianglefree}
If $G$ is a triangle-free graph in $\cF$ and $D$ is a $\Gamma(G)$-set, then $|pn[u,D]|\le 2$ for every $u \in D$.
\end{cor}
In particular, this corollary holds for bipartite graphs in $\cF$.  Note that if $G$ is triangle-free, then $|{\rm pn}[u,D]| = 2$ in the conclusion
of Corollary~\ref{cor:trianglefree} is possible only if $u$ is isolated in the subgraph induced by $D$.

\begin{lem} \label{lem:pn-in-closednhbd}
If $G \in \cF$ and $D$ is any $\Gamma(G)$-set, then for every vertex $u$ in $G$, there exists $x \in D$ such that pn$[x,D]\subseteq N[u]$.
\end{lem}
\begin{proof}
Let $G \in \cF$. Suppose for the sake of contradiction that $G$ has a minimal dominating set $D$ of cardinality $\Gamma(G)$ and there exists $u \in V$ such that
pn$[x,D]-N[u]\neq \emptyset$ for every $x \in D$.  Let $D=\{x_1,\ldots,x_n\}$ and let $x_i' \in {\rm pn}[x_i,D]-N[u]$ for each $i\in[n]$.  The sequence
$(x_1,\ldots,x_n)$ is a Grundy dominating sequence.  Let $S=(u,x_1,\ldots,x_n)$.  It follows  that with respect to $S$, $u$ footprints itself, and for each $i \in [n]$, the
vertex $x_i$ footprints $x_i'$.  That is, $S$ is a legal sequence in $G$ and $|S| > n=\ggr(G)$.  This contradiction establishes the lemma.
\end{proof}

For vertices not belonging to the $\Gamma(G)$-set, the conclusion of Lemma~\ref{lem:pn-in-closednhbd} can be strengthened as follows.

\begin{lem} \label{lem:twoprivatenhbds}
Let $G \in \cF$ and let $D$ be a $\Gamma(G)$-set.  For every $u \in V-D$ there exist
distinct vertices $a$ and $b$ in $D$ such that ${\rm pn}[a,D] \cup {\rm pn}[b,D] \subseteq N[u]$.
\end{lem}
\begin{proof}
Let $S=(x_1,\ldots,x_n)$ be a sequence such that $\widehat{S}=D$ and let $u \in V-D$.  By Lemma~\ref{lem:pn-in-closednhbd}, there exists $i \in [n]$ such that
pn$[x_i,D] \subseteq N[u]$.  Since $S$ is commutative, we may assume that $i=1$.  Suppose first that $ux_1 \notin E$.  Since $\ggr(G)=n$, the sequence $T=(u,x_1,\ldots,x_n)$ is not legal.  Since $x_1$ footprints itself with respect to $T$, there exists $j$ with $2 \le j \le n$ such that $N[x_j] \subseteq N[u] \cup N[x_1] \cup \cdots \cup N[x_{j-1}]$.  Now, pn$[x_j,D]\cap N[\{x_1,\ldots,x_{j-1}\}]=\emptyset$.  It follows that pn$[x_j,D] \subseteq N[u]$.  Therefore, pn$[x_1,D] \cup {\rm pn}[x_j,D] \subseteq N[u]$.
Now suppose that $ux_1 \in E$.  Since $u$ and $x_1$ are not  twins, $N(x_1)-N[u]\neq \emptyset$ or $N(u)-N[x_1] \neq \emptyset$.
If $N(x_1)-N[u]\neq \emptyset$, then $(u,x_1)$ is a legal sequence but $(u,x_1,x_2,\ldots,x_n)$ is not a legal sequence.  As in the first case above, we conclude that there exists $j$ with $2 \le j \le n$, such that pn$[x_1,D] \cup {\rm pn}[x_j,D] \subseteq N[u]$.  On the other hand, if $N(u)-N[x_1] \neq \emptyset$, then $(x_1,u)$ is a legal
sequence but $(x_1,u,x_2,\ldots,x_n)$ is not legal.  Once again, the same reasoning implies that there exists $j$ with $2 \le j \le n$, such that pn$[x_1,D] \cup {\rm pn}[x_j,D] \subseteq N[u]$.
\end{proof}
In the more restricted class $\cFa$ if the set $D$ in Lemma~\ref{lem:twoprivatenhbds} is independent, then $a \in {\rm pn}[a,D]$ for every $a \in D$.
Therefore, we immediately get the following result.

\begin{cor} \label{cor:atleast2}
If $G$ is a graph of order at least $3$ that belongs to $\cFa$ and $A$ is any $\alpha(G)$-set, then $|N(u) \cap A| \ge 2$ for every $u\in V-A$.
\end{cor}

Let $G$ be a graph in $\cF$. Arbitrarily choose and then fix a $\Gamma(G)$-set, $D=\{x_1,\ldots,x_n\}$, and apply the following notation. Let $P_i={\rm pn}[x_i,D]$, where $i\in [n]$, and $P=\cup_{i=1}^n{P_i}$. For any $u\in V-D$ let $I_u=\{j:\, P_j\subseteq N[u]\}$. Set $X=V-(D\cup P)$.

\begin{prop}
\label{prp:cFnecessary}
Let $G$ be a graph in $\cF$ and let $D=\{x_1,\ldots,x_n\}$ be any $\Gamma(G)$-set.
\begin{enumerate}[(i)]
\item If $u\in V-D$, then $|I_u|\ge 2$. In particular, if $u\in P_i$, for some $i\in [n]$, then $i\in I_u$.
\item If $v\in P_i$, for some $i\in [n]$, then for every $j\in I_v-\{i\}$, there exists $k\in I_v$ such that $x_jx_k\in E$.
\item If $w\in X$ and $i\in [n]$ such that $x_ix_j\notin E$ for all $j\in I_w$, then $wx_i\in E$ or $i\notin I_w$.
\item $\ggr(G\langle X\rangle)\le n$.
\end{enumerate}
\end{prop}
\begin{proof}
The first part of the statement (i) is proved in Lemma~\ref{lem:twoprivatenhbds}, while the second part follows from Lemma~\ref{lem:pn-a-clique}. For the proof of the statement (ii) assume that there exist $v\in P_i$ and $j\in I_v-\{i\}$ such that $x_jx_k\notin E$ for all $k\in I_v$. Consider the sequence, which starts with the vertices from $\{x_k:\,k\in I_v-\{j\}\}$ in any order, and is followed by $(v,x_j)$. Clearly, for each $k\in I_v-\{j\}$, $x_k$  footprints the vertices in $P_k$, while $v$ footprints vertices in $P_j$. Finally, $x_j$ footprints itself, since it is not adjacent to  $x_k$, for any  $k\in I_v$. In the end, one can add the remaining vertices of $D$ to the sequence, each of which footprints its private neighborhood. The resulting sequence is legal of length $n+1$, a contradiction to $G\in\cF$.

For the proof of (iii) let $w\in X$ and $i\in [n]$ such that $x_ix_j\notin E$ for all $j\in I_w$, and assume that $wx_i\notin E$ while $i\in I_w$.
Consider the sequence, which starts with the vertices from $\{x_j:\,j\in I_w\}-\{x_i\}$ in any order, is followed by $(w,x_i)$, and completed by the remaining vertices in $D$. This sequence is legal because each vertex of $D-\{x_i\}$ footprints  a vertex in $P$, $w$ footprints vertices in $P_i$, and $x_i$ footprints itself.
Statement (iv) is clear.
\end{proof}

\section{Examples and constructions of graphs in $\cF$}\label{sec:examples}
\label{sec:Examples}

In this section, we present some families of graphs that belong to $\cF$ or $\cFa$ as well as give some constructions by which the class $\cF$ is preserved.
The following classes of graphs belong to $\cF$, most of which also belong to $\cFa$. (Note that some of the graphs in the following classes of graphs have twins, yet the equality $\Gamma(G)=\ggr(G)$ holds for all graphs $G$ in the mentioned classes.)

\begin{enumerate}
\item Complete multipartite graphs, $G=K_{n_1,\ldots,n_k}$, such that $n_1 \ge \cdots \ge n_k$, $k \ge 2$, and $n_{k-1}\ge 2$.   Note that  $\alpha(G)=\Gamma(G)=\ggr(G)=n_1$. The special case where $n_i=2$ for all $i\in [k]$ are the so-called cocktail-party graphs, for which the authors of~\cite{br-go-mi-ra-ri-2014} proved that the Grundy domination number equals the domination number.

\item Prisms over complete graphs, $G= K_n \cp K_2$, for $n \ge 2$.  Note that $\Gamma(G)=\ggr(G)=n$.  Also, $2=\alpha(G)$, which is less than $\ggr(G)$ unless $n=2$. 

\item Certain subclasses of Kneser graphs, as noted by Bre\v{s}ar, Kos and Torres in ~\cite{br-ko-to-2019}. Given positive integers $n$ and $r$ such that $n \ge 2r$, the Kneser graph $K(n,r)$ has as its vertex set the set of all $r$-subsets of $[n]$.  Two vertices  are adjacent in $K(n,r)$ if and only if they are disjoint. A famous result by Erd\"{o}s, Ko and Rado~\cite{er-ko-ra-1961} is that $\alpha(K(n,r))=\binom{n-1}{r-1}$.  Bre\v{s}ar et al. proved that $\ggr(K(n,2))=\alpha(K(n,2))$
for $n \ge 6$ and that for any $r \ge 3$ there exists a positive integer $n_r$ such that $\ggr(K(n,r))=\alpha(K(n,r))=\binom{n-1}{r-1}$, for $n \ge n_r$.  Therefore,
$K(n,2)\in \cFa$ if $n \ge 6$, and for $n \ge 3$, $K(n,r) \in \cFa$ for each $n$ larger than some threshold value that depends on $r$.

\item The class of (twin-free, connected) cographs. Recall that the class of $P_4$-free graphs (also known as {\em cographs}) are those graphs that can be constructed from $K_1$ by repeatedly applying the graph operations of
taking disjoint unions or joins.  It was proved in~\cite{br-go-mi-ra-ri-2014} that $\ggr(G)=\alpha(G)$, for any cograph $G$, implying that the class $\cFa$ contains the class of twin-free connected cographs.
Alternatively, one can prove this by using Lemma~\ref{lem:join} below together with the obvious fact that the disjoint union of graphs from $\cF$ have equal independence and Grundy domination numbers.
\end{enumerate}

In the following result we consider the join $G\oplus H$ of graphs $G$ and $H$.  This is obtained from the disjoint union of $G$ and $H$ by adding the
edges from the set $\{gh:\, g\in V(G) \text{ and } h\in V(H)\}$.
\begin{lem} \label{lem:join}
If $G_1$ and $G_2$ are graphs in $\cF$, then $G_1 \oplus G_2 \in \cF$.
\end{lem}
\begin{proof}
Let $G_1$ and $G_2$ be graphs in $\cF$ and assume without loss of generality that $\Gamma(G_1) \ge \Gamma(G_2)$.  We claim that $\Gamma(G_1 \oplus G_2)=\ggr(G_1 \oplus G_2)=\Gamma(G_1)$.  Note that a minimal dominating set of either $G_1$ or $G_2$ is a minimal dominating set of their join.      Furthermore, if $A$ is any subset of $V(G_1 \oplus G_2)$ and $A$ contains a vertex from each of $G_1$ and $G_2$,  then $A$ dominates $G_1 \oplus G_2$.  We infer that $\Gamma(G_1 \oplus G_2)\ge \Gamma(G_1)$.
If $\Gamma(G_1)=1$, then the join is a complete graph and the claim holds.  Suppose now that $\Gamma(G_1)\ge 2$.  By the above we can select  a largest minimal dominating set $D$ of $G_1 \oplus G_2$ to be a $\Gamma(G_1)$-set.  Any permutation of the vertices of $D$ is a dominating sequence of $G_1 \oplus G_2$.  Any sequence of vertices from
$V(G_1 \oplus G_2)$ of length more than $|D|$ is not a legal sequence since $\ggr(G_1)=|D|$ and since a sequence that contains at least one vertex from each of $G_1$ and $G_2$
cannot be extended to a legal sequence.
\end{proof}

It seems natural to investigate whether there exist nontrivial Cartesian products in $\cF$ or in $\cFa$.  We will make use of the following result from~\cite{br-bu-go-kl-ko-pa-tu-vi-2016}.

\begin{prop} {\rm (\cite[Proposition 3]{br-bu-go-kl-ko-pa-tu-vi-2016})} \label{prop:CartesianGrundy}
For any two graphs $G$ and $H$,
\[\ggr(G \cp H) \ge \max \{\ggr(G)n(H),\ggr(H)n(G)\}\,.\]
\end{prop}

We now prove a necessary condition for a Cartesian product to belong to $\cFa$.
\begin{lem} \label{lem:CartesianGrundyAlpha}
If $G$ and $H$ are two graphs such that $G\cp H \in \cFa$, then both $G$ and $H$ are in $\cFa$ and $\frac{\alpha(G)}{n(G)}=\frac{\alpha(H)}{n(H)}$.
\end{lem}

\begin{proof}
Let $G$ and $H$ be graphs such that $G \cp H \in \cFa$. It is well-known that $\alpha(G \cp H) \le \min\{\alpha(G)n(H),\alpha(H)n(G)\}$.
By Proposition~\ref{prop:CartesianGrundy} we have $\ggr(G \cp H) \ge \max \{\ggr(G)n(H),\ggr(H)n(G)\}$.  Since $\ggr(G \cp H)= \alpha(G \cp H)$, it follows that
\[ \max \{\ggr(G)n(H),\ggr(H)n(G)\} \le \min\{\alpha(G)n(H),\alpha(H)n(G)\}\,.\]  We
infer the following.
\begin{eqnarray}
 \ggr(G)n(H) & \le \alpha(G)n(H) \\
 \ggr(G)n(H) & \le \alpha(H)n(G)\\
 \ggr(H)n(G) & \le \alpha(G)n(H)\\
 \ggr(H)n(G) & \le \alpha(H)n(G)
\end{eqnarray}
The first of these inequalities together with $\alpha(G) \le \ggr(G)$ implies that $\ggr(G)=\alpha(G)$.  Similarly, using the last
of these four inequalities we get $\ggr(H)=\alpha(H)$.  Therefore, $\{G,H\} \subseteq \cFa$.  Finally, we use the second and the third of these inequalities
together with $\ggr(G)=\alpha(G)$ and $\ggr(H)=\alpha(H)$ to conclude
that $\frac{\alpha(G)}{n(G)}=\frac{\alpha(H)}{n(H)}$.
\end{proof}

By Lemma~\ref{lem:CartesianGrundyAlpha}, all but one book graph, or graphs of the form $K_{1, m} \cp K_2$, are not in $\cFa$, since $\frac{\alpha(K_2)}{n(K_2)} = \frac{1}{2} < \frac{\alpha(K_{1,m})}{n(K_{1,m})}$ when $m>1$. Since $P_n \in \cFa$ if and only if $n\in [3]$, one can use Lemma~\ref{lem:CartesianGrundyAlpha} together with the fact
that $\alpha(P_3\cp P_3) =5$ and $\ggr(P_3\cp P_3)=6$ to see that $P_2\cp P_2$ is the only grid graph with two nontrivial factors that is in $\cFa$.
The $3$-dimensional hypercube $Q_3=C_4 \cp K_2$ is an example of a (nontrivial) Cartesian product that belongs to $\cFa$. In addition, all hypercubes belong to $\cFa$, which we will prove by using some connections with linear algebra. 

Let $G$ be a graph of order $n$, and without loss of generality denote its vertex set by $[n]$. Let ${\cal S}(G)$ be the family of all $n \times n$ real symmetric matrices whose $i, j$-entry, where $i\ne j$, is non-zero if and only if  $ij\in E(G)$. Note that there are no restrictions on the diagonal entries. {\em Minimum rank} of $G$ is defined as ${\rm mr}(G)=\min\{{\rm rank}(A):\, A\in {\cal S}(G)\}$. 

In~\cite{br-bu-go-kl-ko-pa-tu-vi-2017}, a close connection was established between a variation of the Grundy domination number, called the Z-Grundy domination number, and the zero forcing number, the concept introduced in~\cite{AIM} and studied in a number of papers both by graph theorists and linear algebraists. Lin continued the investigation from~\cite{br-bu-go-kl-ko-pa-tu-vi-2017}, and among other results found a similar relation between the Grundy domination number of a graph and the so-called loop zero forcing number. The latter concept is in turn related to a version of a minimum rank of a graph, which is defined as follows. 

Let ${\cal S}_{\dot{\ell}}(G)$ denote the set of all matrices in ${\cal S}(G)$ whose all diagonal entries are non-zero. Then, ${\rm mr}_{\dot{\ell}}(G)=\min\{{\rm rank}(A):\, A\in {\cal S}_{\dot{\ell}}(G)\}$. Lin proved that $\ggr(G)\le{\rm mr}_{\dot{\ell}}(G)$ holds for every graph $G$, which together with  \eqref{eqn:expandedstring} yields $$\alpha(G)\le \ggr(G)\le{\rm mr}_{\dot{\ell}}(G),$$
for any graph $G$. Now, let $G$ be the hypercube $Q_d$, where $d$ is a positive integer. Clearly, $\alpha(Q_d)=2^{d-1}$, which gives $2^{d-1}\le\ggr(Q_d)$. For the reversed inequality we invoke a result of Huang, Chang and Yeh from~\cite[Theorem 10]{chang}, where in the proof a matrix $B_d$ appears, which belongs to ${\cal S}(Q_d)$. In addition, it is easy to see that diagonal entries of $B_d$ are non-zero, which implies $B_d\in {\cal S}_{\dot{\ell}}(Q_d)$. It is proved in~\cite{chang} that ${\rm rank}(B_d)=2^{d-1}$, which yields ${\rm mr}_{\dot{\ell}}(Q_d)\le 2^{d-1}$, hence $\ggr(Q_d)\le 2^{d-1}$ . We thus infer the following  result.

\begin{prop}
Hypercubes belong to $\cFa$. More precisely, $\ggr(Q_d)=\alpha(Q_d)=2^{d-1}$ for all positive integers $d$. 
\end{prop}  

This result is an improvement of the result from~\cite{AIM} that the zero-forcing number in hypercubes $Q_d$ equals $2^{d-1}$.

\section{The Class $\cFa$}
\label{sec:cFa}


In this section, we prove our main result, a characterization of the graphs that are in $\cFa$. In the beginning of the section we focus on triangle-free graphs that are in $\cFa$. Our first result shows that classifying triangle-free graphs in $\cFa$ reduces to classifying all bipartite graphs in $\cFa$.

\begin{prop}\label{prop:bipartite} If $G$ is a triangle-free graph of order at least $3$ and  $G \in \cFa$, then $G$ is bipartite with $\alpha(G)=n(G)/2$ or $G$ is bipartite and has a unique $\alpha(G)$-set. In particular, if $A$ is any $\alpha(G)$-set, then $V-A$ is independent.
\end{prop}

\begin{proof}
Let $A=\{x_1,\ldots,x_n\}$ be an $\alpha(G)$-set.  We first show that if $x,y\in V-A$ such that $N(x) \cap N(y) \cap A = \emptyset$, then $xy \not\in E$. Suppose to the contrary that  $\{x,y\} \subseteq V-A$ such that $xy\in E$ and $N(x) \cap N(y) \cap A=\emptyset$.  Assume without loss of generality that $N(x)\cap A=\{x_1,\ldots,x_r\}$ and $N(y)\cap A =\{x_{r+1},\ldots, x_t\}$.  The sequence $S=(x_1,\ldots, x_n)$ is a Grundy dominating sequence.  But now $S'=(x_1,\ldots,x_r,x,x_{r+1},\ldots,x_n)$ is a legal sequence.  (In $S'$, $x$ footprints $y$ and $x_i$ footprints $x_i$ for every $i \in [n]$.)  This is a contradiction and therefore no such pair $x,y \in V-A$ exists.

Next, we show that $V-A$ is indeed independent. Let $\{x,y\}\subseteq V-A$.  If $N(x) \cap N(y) \cap A=\emptyset$, then $xy \notin E$ from the above argument.  On the other hand, if $N(x) \cap N(y) \cap A \neq \emptyset$, then $xy \notin E$ since $G$ is triangle-free.
In all of these cases $V-A$ is an independent set. It follows that $G$ is bipartite. Thus, $\alpha(G)\ge n(G)/2$. 

Now, if $\alpha(G)=n(G)/2$, we are done, so let us assume that $\alpha(G) >n(G)/2$. We claim that there is only one $\alpha(G)$-set, and suppose to the contrary that $A$ and $B$ are distinct $\alpha(G)$-sets.  Let $C=A \cap B$ and let $D=V-(A \cup B)$.  In addition, let $A_1=A-C$ and $B_1=B-C$.   Since $A\neq B$ and
$|A|=|B|=\alpha(G) > n(G)/2$, we infer that $C\neq \emptyset$ and that neither of $A_1$ nor $B_1$  is empty.  From the above argument, both of $A_1\cup D$
and $B_1 \cup D$ are independent.  Let $u$ be any vertex in $C$ and let $v$ be any vertex in $A_1$.  Since $N(C) \subseteq D$ and $N(D) \subseteq C$ there is no
$u,v$-path in $G$, which contradicts the fact that $G$ is connected.  Therefore, if  $\alpha(G) >n(G)/2$, then $G$ has a unique $\alpha(G)$-set.
\end{proof}

Based on the above result, we spend the remainder of this section focusing on bipartite graphs. We next give two  properties that help us determine when a bipartite graph is in $\cFa$.


\medskip
\noindent {\bf Property H:}  If $A$ is any $\alpha(G)$-set, then for every $W\subseteq V-A$ with $|W|<|A|$, we have $|N(W)|\ge |W|+1$.

\smallskip

\begin{prop} \label{prop:PropertyH}
If $G$ is a bipartite graph of order at least $3$ and $G \in \cFa$, then $G$ satisfies Property H.
\end{prop}

\begin{proof}
Assuming that the statement is false, let $A$ be an $\alpha(G)$-set, and
$W=\{w_1,\ldots,w_k\}$ be a subset of $V-A$ with $|W|<|A|$ and $|N(W)|\le |W|$. Since $G$ is connected, there exists $x\in N(W)$ having a neighbor $y\notin W$. Letting $A-N(W)=\{x_1,\ldots,x_\ell\}$, note that $S=(w_1,\ldots,w_k,x,x_1,\ldots, x_\ell)$ is a legal sequence, since $x$ footprints $y$ while every other vertex in $S$ footprints itself (note that $W$ is independent by Proposition~\ref{prop:bipartite}). Since $|\widehat{S}|\ge|A|+1$, this contradicts the assumption that  $G\in\cFa$.
\end{proof}

\noindent {\bf Property T:} If $A$ is any $\alpha(G)$-set and $w \in V - A$, then for each $u \in N(w) \cap A$, we have $N(u) \subseteq N((N(w)\cap A) - \{u\})$.

\smallskip

\vskip5mm
\begin{prop} \label{prop:PropertyT}
If $G$ is a triangle-free graph of order at least $3$ and $G \in \cFa$, then $G$ satisfies Property T.
\end{prop}

\begin{proof}
Let $A = \{x_1, \dots ,x_n\}$ be an $\alpha(G)$-set and let $w \in V - A$. By Proposition~\ref{prop:bipartite}, $G$ is bipartite and $V-A$ is independent.
Reindexing if necessary, we may assume $N(w)=N(w)\cap A = \{x_1, \dots, x_k\}$. Suppose for some $i \in [k]$ that
\[N(x_i) \not\subseteq N(N(w) - \{x_i\}).\]
That is, there exists $y \in N(x_i)$ such that $\displaystyle y \not\in \bigcup_{j\ne i, j \in [k]} N(x_j)$. Note that
\[S = (x_1, \dots, x_{i-1},x_{i+1}, \dots, x_k, w, x_i, x_{k+1}, \dots, x_n)\]
is a legal sequence since $x_j$ footprints itself for $j \in [n]-\{i\}$, $w$ footprints $x_i$, and $x_i$ footprints $y$. However, this is a contradiction.
Therefore no such $x_i$ exists, and $G$ satisfies Property T.
\end{proof}

\begin{lem}\label{lem:PropP} If $G$ is a bipartite graph of order at least $3$ in $\cFa$, and $A$ is any $\alpha(G)$-set, then $|N(x) \cap N(y) \cap A| \ne 1$ for every pair $x,y \in V-A$.
\end{lem}

\begin{proof}
By Proposition~\ref{prop:PropertyT}, $G$ satisfies Property T. By Corollary~\ref{cor:atleast2}, $|N(u)\cap A| \ge 2$ for every $u \in V-A$. Let $\{x, y\} \subseteq V-A$ and suppose to the contrary that $N(x) \cap N(y) \cap A = \{z\}$. It follows that $y \in N(z)$ and $y \not\in N((N(x) \cap A) - \{z\})$, which violates Property T.
\end{proof}

By Corollary~\ref{cor:atleast2} and Lemma~\ref{lem:PropP}, we can classify all bipartite graphs in $\cF$ of girth $6$ or more.

\begin{thm} \label{thm:bipartitegirth6}
A bipartite graph $G$ with girth at least $6$ is in $\cF$ if and only if $G=K_1$ or $G$ is a star $K_{1,r}$, where $r\ge 1$.
\end{thm}
\begin{proof}
It is easy to see that $K_1$ and all stars belong to $\cFa$.  For the converse let $G$ be a bipartite graph in the class $\cF$ having order
at least $3$ and girth at least $6$.  Since $G$ is bipartite, it follows from Theorem~\ref{thm:topthreebipartite} that $\alpha(G)=\Gamma(G)$.  Thus,
$G \in \cF$ if and only if $G \in \cFa$.
Let $A$ be any $\alpha(G)$-set.   Suppose that $|V-A| \ge 2$.  By Proposition~\ref{prop:bipartite}, $V-A$ is independent.
Since $G$ is connected, there exists a pair of vertices $x$ and $y$ that belong to $V-A$ such that $N(x) \cap N(y) \cap A=N(x) \cap N(y) \neq \emptyset$.  By
Lemma~\ref{lem:PropP}, we infer that  $|N(x) \cap N(y) \cap A| \ge 2$, which implies that $G$ contains a $4$-cycle.  This is a contradiction, and thus
$|A|=n-1$, which means that $G$ is a star.
\end{proof}

\begin{prop}
 Let $G$ be a connected bipartite graph of order $5$ or more, and let $A$ be an $\alpha(G)$-set. If there exist distinct vertices $x$ and $y$ in $V - A$
 and distinct vertices $u$ and $v$ in $A$ such that $N(x) \cap A= \{u, v\} = N(y) \cap A$, then $G \not\in \cFa$.
\end{prop}

\begin{proof}
Suppose to the contrary that $G \in \cFa$. Note that $S = (x_1, x_2, \dots, x_k)$ is a legal sequence where $\widehat{S} = A$. Moreover, reindexing if necessary, we may assume $N(x) \cap A = \{x_1, x_2\} = N(y)\cap A$. Note that if $N(\{x_1,x_2\})=\{x,y\}$, then $G = C_4$. Therefore, we may assume there exists $z \in N(x_1) - \{x,y\}$.

Suppose first that $N(x_2) \subseteq N(x_1)$. If there exists $w \in N(x_1) - N(x_2)$, then $S' = (x_2, x, x_1, x_3,x_4,  \dots, x_k)$ is a legal sequence as $x_i$ footprints itself for $2 \le i \le k$, $x_1$ footprints $w$, and $x$ footprints $x_1$. Thus, this case cannot occur and we may assume $N(x_1) = N(x_2)$. However, now $S'' = (x, y, x_1, x_3, x_4, \dots, x_k)$ is a legal sequence since each vertex of $S''$ other than $x_1$ footprints itself and $x_1$ footprints $z$. Therefore, we may assume $N(x_2) \not\subseteq N(x_1)$. Let $t \in N(x_2) - N(x_1)$ and consider $T = (x_1, x, x_2, x_3, \dots, x_k)$. Note that $T$ is a legal sequence as each vertex of $T$ other than $x$ and $x_2$ footprint themselves, $x$ footprints $x_2$, and $x_2$ footprints $t$.

In each case, we have a contradiction. Thus, $G \not\in \mathcal{F}_{\alpha}$.
\end{proof}

We next present a family of bipartite graphs that shows Property H and T alone are not sufficient to guarantee that a bipartite  graph is in $\cFa$. For each positive
integer $n$ at least $3$, we construct a bipartite graph $G_n$.  The set of vertices of $G_n$ consists of two partite sets $A=\{a_{i,j}:\,i\in [n], j\in[3]\}\cup\{x,y\}$ and $B=\{u_i:\, i\in [n]\}\cup \{w_i:\, i\in [n]\}$. The vertices $x$ and $y$ are adjacent to every vertex of $B$, and for every $i\in [n]$, the vertex $a_{i,1}$ is adjacent to $u_i$ and $w_i$, while $a_{i,2}$ and $a_{i,3}$ are adjacent only to $u_i$. Now, $A$ is the unique $\alpha$-set of $G$ with $|A|=3n+2$, and
$G$ satisfies Properties H and T.  On the other hand, the sequence $S=(a_{1,3},a_{1,2},u_1,a_{1,1},a_{2,3},a_{2,2},u_2,a_{2,1},\ldots,a_{n,3},a_{n,2},u_n,a_{n,1})$ is a closed neighborhood sequence, since $a_{i,3}$ and $a_{i,2}$ footprint themselves, $u_i$ footprints $a_{i,1}$, and $a_{i,1}$ footprints $w_i$, for all $i\in [n]$. Since $|\widehat{S}|=4n$, this implies that $G\notin \cFa$.

We next consider the following additional property, which can be viewed as a generalization of Property T.

\vskip5mm
\noindent {\bf Property $\rm{T}^*$:} If $A$ is any $\alpha(G)$-set and $W \subseteq V - A$, then for every $U \subseteq N(W)\cap A$ with $|U|=|W|$ we have $N(U) \subseteq N((N(W)\cap A) - U)$.

\medskip

To see that Property $\rm{T}^*$ is a generalization of Property T, note that  the latter is obtained from Property $\rm{T}^*$ by letting $W=\{w\}$ and $U=\{u\}$.

\begin{prop} \label{prp:sufficient}
Let $G$ be a connected bipartite graph such that for any $\alpha(G)$-set $A$, the set $V-A$ is independent. If $G$ satisfies Properties H and $\rm{T}^*$, then $G\in\cFa$.
\end{prop}

\begin{proof}
Let $G$ be a bipartite graph in which Properties H and $\rm{T}^{*}$ hold (for every $\alpha$-set of $G$). Let $A$ be an $\alpha(G)$-set. Suppose that $S$ is a closed neighborhood sequence of $G$ that involves some vertices in $V-A$, and let $w_1,\ldots,w_k$ be these vertices in the order that they appear in $S$. We may set $$S=(a_1,\ldots,a_{i_1},w_1,a_{i_1+1}, \ldots,a_{i_2},w_2,\ldots,a_{i_k},w_k,a_{i_k+1},\ldots,a_m),$$ where each of $i_1=0$  and $i_k=m$ is also possible (in the first case $w_1$ is the first vertex of $S$, and in the second case $w_k$ ends $S$), and $a_r\in A$ for all $r\in [m]$.

We claim that for each $p\in [k]$,
\begin{equation}
\label{eq:neighbors-j}
|(N(w_1)\cup\cdots\cup N(w_p))- \{a_1,\ldots,a_{i_p}\}|\ge p.
\end{equation}
First, we prove this for $p\in\{1, 2\}$ and then use induction. For $p=1$, suppose to the contrary, that $|N(w_1)- \{a_1,\ldots,a_{i_1}\}|=0$. This implies that every vertex in $N[w_1]$ is dominated by $\{a_1,\ldots,a_{i_1}\}$. Hence, $w_1$ footprints no vertex, a contradiction. For $p=2$, suppose to the contrary that $|(N(w_1)\cup N(w_2)) - \{a_1, \dots, a_{i_2}\}| \le 1$. Note that we may assume $|(N(w_1)\cup N(w_2)) - \{a_1, \dots, a_{i_2}\}| = 1$ for otherwise $w_2$ does not footprint a vertex. Let $(N(w_1)\cup N(w_2)) - \{a_1, \dots, a_{i_2}\}= \{a\}$. By Property H, $|N(w_2)| \ge 2$ and therefore some neighbor of $w_2$ is in $\{a_1, \dots, a_{i_2}\}$. Thus, $w_2$ footprints only $a$ and so $aw_1 \not\in E$. It follows that $N(w_1) \subseteq \{a_1, \dots, a_{i_2}\}$. Let $x$ be the last vertex of $N[w_1]$ to appear in $(a_1, \dots, a_{i_1},w_1,a_{i_1+1}, \dots, a_{i_2})$. It is clear that $x \ne w_1$, for otherwise $w_1$ does not footprint a vertex. Hence, $x \in A$, and so by Property $\rm{T}^{*}$, $N(x) \subseteq N(N(w_1)- \{x\})$. We derive that $x$ does not footprint a vertex, a contradiction.

Now, assume that for some $j\ge 2$
\[|(N(w_1)\cup\cdots\cup N(w_{j-1}))-  \{a_1,\ldots,a_{i_{j-1}}\}|\ge j-1\,.\]
We claim that also
\begin{equation}\label{eq:intermediate}
|(N(w_1)\cup\cdots\cup N(w_{j-1})) -  \{a_1,\ldots,a_{i_{j}}\}|\ge j-1.
\end{equation}
Suppose this is not the case, and let $t\in\{i_{j-1}+1,\ldots,i_j\}$ be the smallest index such that
\[|(N(w_1)\cup\cdots\cup N(w_{j-1})) -  \{a_1,\ldots,a_{t}\}| = j-2\,.\]
Let $w\in V-A$ be a vertex footprinted by $a_t$. Let $W=\{w_1,\ldots, w_{j-1}\}$ and let
$U=(N(w_1)\cup\cdots\cup N(w_{j-1})) -  \{a_1,\ldots,a_{t-1}\}$. Then, $|U|=j-1=|W|$, $w\in N(U)$, while $w\notin N(N(W)-U)$, which is a contradiction to Property $\rm{T}^*$, and~\eqref{eq:intermediate} is proved.

Since $S$ is a closed neighborhood sequence, $w_j$ footprints at least one vertex.  We claim that $w_j$ footprints a vertex other than itself.  If this were not
the case, then $N[w_j]-N[\widehat{S'}]=\{w_j\}$, where $S'$ is the (leading) subsequence of $S$ given by
$S'=(a_1,\ldots,a_{i_1},w_1,a_{i_1+1}, \ldots,a_{i_2},w_2,\ldots,a_{i_j})$.
However, this is not possible since $\widehat{S'}$ dominates $w_j$.  Thus, $w_j$ footprints some vertex $a \in N(w_j)$.
It is clear that $a\notin \{a_1,\ldots,a_{i_j}\}$, which implies, combined with~\eqref{eq:intermediate} that
$|(N(w_1)\cup\cdots\cup N(w_j)) -  \{a_1,\ldots,a_{i_j}\}|\ge j$.  By induction we now have that
\[|(N(w_1)\cup\cdots\cup N(w_p))- \{a_1,\ldots,a_{i_p}\}|\ge p, \text{ for every } p\in [k]\,,\]
and so \eqref{eq:neighbors-j} is proved.
In particular, $|(N(w_1)\cup\cdots\cup N(w_k)) -  \{a_1,\ldots,a_{i_k}\}|\ge k.$
We claim that
\begin{equation}
\label{eq:neighbors-k}
|(N(w_1)\cup\cdots\cup N(w_k)) - \{a_1,\ldots,a_{m}\}|\ge k.
\end{equation}
Note that whenever a vertex in $(N(w_1)\cup\cdots\cup N(w_{k}))\cap \{a_{i_k+1},\ldots,a_m\}$ is added to $S$, it does not footprint itself.
Suppose that $|(N(w_1)\cup\cdots\cup N(w_k)) - \{a_1,\ldots,a_{m}\}|< k$, and let $a_t\in(N(w_1)\cup\cdots\cup N(w_{k}))\cap \{a_{i_k+1},\ldots,a_m\}$ be the vertex with the smallest index $t$ such that
$|(N(w_1)\cup\cdots\cup N(w_k)) -  \{a_1,\ldots,a_{t}\}|= k-1$. Let $w\in V-A$ be a vertex footprinted by $a_t$. Now, setting $W=\{w_1,\ldots,w_k\}$, and
$U=(N(w_1)\cup\cdots\cup N(w_k)) -  \{a_1,\ldots,a_{t-1}\}$, we infer that $|U|=k$, and $w\in N(U)$, while $w\notin N(N(W)-U)$. This is a contradiction with Property $\rm{T}^*$, hence~\eqref{eq:neighbors-k} holds. Since $V-A$ is independent by the initial assumption, we get $$|(N(w_1)\cup\cdots\cup N(w_k))\cap A - \{a_1,\ldots,a_{m}\}|\ge k.$$
Thus
\[|\widehat{S}| = m+k \le |\widehat{S}\cap A| + |(N(w_1)\cup\cdots\cup N(w_k))\cap A - (\widehat{S}\cap A)| \le |A|\] and so $G\in\cFa$.
\end{proof}

We now provide an example of a graph which is in $\cFa$ yet does not satisfy Property $\rm{T}^*$. From Section~\ref{sec:examples}, $Q_3 \in \cFa$. However, $Q_3$ does not satisfy Property $\rm{T^*}$. For example, in Figure~\ref{fig:Q_3} consider the $\alpha(Q_3)$-set depicted by the black vertices and the sets $W = \{1, 2,3\}$ and $U = \{a, c, d\}$. Then $3 \in N(U)$ yet $3 \not\in N(N(W) - U)$.

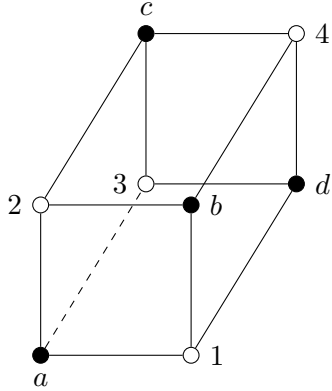
\begin{figure}[h!]
\begin{center}
\begin{tikzpicture}[scale=2]

	\vertex (1) at (0, 0) [fill=black,label=below:$a$]{};
	\vertex (2) at (1, 0) [label=right:$1$]{};
	\vertex (3) at (0, 1) [label=left:$2$]{};
	\vertex (4) at (1, 1) [fill=black,label=right:$b$]{};
	\vertex (5) at (.7,1.14) [label=left:$3$]{};
	\vertex (6) at (1.7,1.14) [fill=black,label=right:$d$]{};
	\vertex (7) at (.7,2.14) [fill=black,label=above:$c$]{};
	\vertex (8) at (1.7,2.14) [label=right:$4$]{};
	\path
	
	(1) edge (2)
	(1) edge (3)
	(2) edge (4)
	(3) edge (4)
	(5) edge (6)
	(5) edge (7)
	(6) edge (8)
	(7) edge (8)
	(1) edge[dashed] (5)
	(2) edge (6)
	(3) edge[] (7)
	(4) edge (8)
	
	;

\end{tikzpicture}
\end{center}
\caption{Graph $Q_3$ does not satisfy Property $\rm{T^*}$.}
\label{fig:Q_3}
\end{figure}

We point out that Property H and Property $\rm{T}^*$ are structural properties. Although we were not able to find structural properties that are necessary and sufficient to guarantee that a graph is in $\cFa$, we are able to show that the  following property based on checking all legal sequences of $G$ is necessary and sufficient to guarantee that a graph is in $\cFa$.

\medskip

\noindent {\bf Property U:} Let $A$ be any $\alpha(G)$-set. If $S$ is a legal sequence and $W = \widehat{S} \cap (V-A)$, then $|(N(W) \cap A) - \widehat{S}| \ge |W|$.

\smallskip

\begin{thm}
\label{thm:main}
If $G$ is a connected graph, then $G \in \cFa$ if and only if $G$ satisfies Property U.
\end{thm}

\begin{proof} Let $S = (x_1, \dots, x_m)$ be a Grundy sequence of $G$, let $A$ be any $\alpha(G)$-set, and let $B = V-A$.
We first show that if $G$ satisfies Property U, then $G \in \cFa$.  We let $W = B \cap \widehat{S}$ and $(N(W) \cap A) \cap \widehat{S} = X$. Let $A' = A - N(W)$. Suppose some $a' \in A'$ is not in $\widehat{S}$. Then $S' = (x_1, \dots, x_m, a')$ is a longer legal sequence, which is a contradiction. Hence, $A' \subseteq \widehat{S}$. By Property U, $|(N(W) \cap A) - \widehat{S}| \ge |W|$ and therefore
\[|\widehat{S}| = |B \cap \widehat{S}| + |A \cap \widehat{S}| = |W| + |X| + |A'| \le |(N(W) \cap A) - \widehat{S}| + |X| + |A'| = |A|.\]
On the other hand, $|\widehat{S}| \ge |A|$ and so it must be that $|\widehat{S}| = |A|$ and $G \in \cFa$.

For the converse, suppose  $G$ does not satisfy Property U.  Thus, there exists a legal sequence $S = (x_1, \dots, x_m)$ where $W = B \cap \widehat{S}$ and $|(N(W) \cap A) - \widehat{S}| < |W|$. Let $A' = A - N(W)$. Let $A''=\{a \in A'\,:\, a \notin \widehat{S}\}$.  Write $A'' = \{y_1, \dots, y_{\ell}\}$. Then we can extend $S$ to the legal sequence $S' = (x_1, \dots, x_m, y_1, \dots, y_{\ell})$. Therefore,
\begin{eqnarray*}
\ggr(G) \ge |\widehat{S'}| &=& |W| + |(N(W) \cap A) \cap \widehat{S}| + |A'| \\
&>& |(N(W) \cap A )- \widehat{S}| + |(N(W) \cap A) \cap \widehat{S}| + |A - N(W)| \\
&=& |A|,
\end{eqnarray*}
and we conclude that $G \notin \cFa$.
\end{proof}

We can use Property U to show that the only $n$-crossed prism graph in $\cFa$ is the $4$-crossed prism graph. Recall the $n$-crossed prism graph for even positive integer $n\ge 4$ is defined as follows. Take two disjoint copies of $C_n$, say $C_n^1 = u_1u_2\cdots u_n$ and $C_n^2 = v_1v_2\cdots v_n$ and add the edges $v_su_{s+1}$ for $s \in \{1, 3, \dots, n-1\}$ and the edges $v_tu_{t-1}$ for $t \in \{2, 4, \dots, n\}$. The $4$-crossed prism graph is isomorphic to the $3$-dimensional hypercube $Q_3$.  Note that the $n$-crossed prism graph is bipartite, cubic, and vertex-transitive.

\begin{cor} The $n$-crossed prism graph is in $\cFa$ if and only if $n=4$.
\end{cor}

\begin{proof} Let $G_n$ be the $n$-crossed prism graph, and let $A = \{u_i, v_i: \text{$i$ is odd}\}$ and $B = \{u_i, v_i: \text{$i$ is even}\}$ be the two $\alpha(G_n)$-sets. We first show that the $n$-crossed prism graph $G_n$ does not satisfy Property U when $n >4$. Consider the legal sequence $S = (u_2, v_1, v_2, u_3, u_1)$.
Thus, $W = B \cap \widehat{S} = \{u_2, v_2\}$ and $(N(W) \cap A) - \widehat{S} = \{v_3\}$. Hence, $G_n$ does not satisfy Property U when $n>4$.
Since $G_4=Q_3$, the converse follows from the result mentioned in Section~\ref{sec:examples}.
\end{proof}

Finally, we note that there is some connection to studying graphs containing triangles in  $\cFa$ and studying bipartite graphs in $\cFa$.
In what follows, we let $G_{uv}$ denote the graph obtained from $G$ be identifying two vertices $u$ and $v$ of $G$ and then removing any duplicate edges that
result from this identification.

\begin{thm}  \label{thm:vertexidentification}
Suppose $G\in \cFa$ and $I$ is a maximum independent set in $G$. For any pair $x, y \in V - I$, $\alpha(G_{xy}) = \ggr(G_{xy})$.
\end{thm}

\begin{proof} Write $I = \{v_1, \dots, v_k\}$ and note that $S = (v_1, v_2, \dots, v_k)$ is a legal sequence. Fix $x, y \in V - I$ and let $w$ denote the vertex of $G' = G_{xy}$ that arises from identifying $x$ and $y$. Note that $\alpha(G') = k$ and that $S$ is a legal sequence in $G'$ since each vertex of $S$ footprints itself. Thus, $\ggr(G') \ge k$. Suppose there exists a legal sequence, say $A = (t_1, t_2, \dots, t_{k+1})$, in $G'$ of length $k+1$. For each $i \in [k+1]$, there is a nonempty
subset $U_i$ of $V(G')$ such that $t_i$ footprints each vertex of $U_i$ with respect to $A$.

Suppose by contradiction that $w \not\in \widehat{A}$. 
If $U_i-\{w\} \neq \emptyset$ for each $i \in [k+1]$, then every
vertex of $A$ footprints at least one vertex in $G$ and hence $A$ is a legal sequence in $G$.  This contradicts the fact that $\ggr(G)=k$.  Thus, $U_j=\{w\}$ for some $j$ with $2 \le j \le k+1$.
Without loss of generality we may assume that $xt_j \in E$.  Now, as a sequence in $G$ we see that $A$ is legal since $t_i$ footprints $U_i$ for $i \neq j$ and $t_j$
footprints $x$.  This again contradicts $\ggr(G)=k$ and therefore we infer that $w \in \widehat{A}$.  That is, $w = t_i$ for some $i \in [k+1]$.  Consider the sequences
$A' = (t_1, t_2, \dots, t_{i-1}, x, t_{i+1}, \dots, t_{k+1})$ and $A'' = (t_1, t_2, \dots, t_{i-1}, y, t_{i+1}, \dots, t_{k+1})$ in $G$.  If $w$ footprints itself with respect to $A$ in $G'$, then $x$ footprints itself with respect to $A'$ in $G$.  Otherwise, $w$ footprints a vertex $s_i \in U_i-\{w\}$ with respect to $A$.  If $xs_i \in E$,
then $x$ footprints $s_i$ in $G$ with respect to $A'$.  Otherwise, $ys_i \in E$ and $y$ footprints $s_i$ in $G$ with respect to $A''$.
It follows that $A'$ or $A''$  is a legal sequence in $G$, which contradicts $\ggr(G)=k$.

Therefore, $\ggr(G') = k = \alpha(G')$.
\end{proof}

By Theorem~\ref{thm:vertexidentification}, if we contract all adjacent pairs of vertices in the complement of an $\alpha(G)$-set the resulting graph will be a bipartite graph
in the class $\cFa$ if the original graph is in $\cFa$.  Thus, a characterization of the bipartite graphs in $\cFa$ gives partial information about the structure of all
graphs in $\cFa$.


\section{Concluding remarks}

In this paper, we initiated the study of graphs $G$ in which $\ggr(G)=\Gamma(G)$, or $\ggr(G)=\alpha(G)$, respectively. Since the graphs $G$ in which $\ggr(G)=\gamma(G)$ have been completely characterized~\cite{br-go-mi-ra-ri-2014,er-2019}, studying the mentioned two classes of graphs is the natural step forward. 

We found several properties of graphs $G$ in family $\cal F$ of connected twin-free graphs with $\ggr(G)=\Gamma(G)$; the properties are related to the partition of a graph $G$ derived from a $\Gamma(G)$-set $D$, which is formed by $D$, private neighborhoods of vertices in $D$, and the remainder of the graph. 
It would be interesting to know if these properties together imply that the graph belongs to $\cal F$, which we formulate as the following problem. 
 
\begin{prob} \label{prob:1} Is the reverse direction in Proposition~\ref{prp:cFnecessary} also true? That is, do the properties (i)-(iv) together imply that a twin-free, connected graph is in $\cal F$?
\end{prob}

Besides the class $\cal F$ we also consider the class $\cFa$ consisting of connected, twin-free graphs with $\alpha(G)=\ggr(G)$.
The two classes of graphs are rather rich, which is reflected in a number of families that belong to one or both of the classes (for instance, we proved that a large family of Kneser graphs as well as all cographs and hypercubes satisfy $\alpha=\ggr$). In addition, several graph operations preserve the property of being in one of the two classes. 

The most thorough investigation was given to the class $\cFa$. We proved that triangle-free graphs in $\cFa$ are always bipartite graphs in which an $\alpha$-set is either unique or of the size half the order. We found two structural properties (called Property H and Property T) of bipartite graphs in $\cFa$, but they turned out not to be sufficient for a graph to be in $\cFa$.   It would be interesting to investigate whether Property T could be strengthened in such a way that together with Property H it would yield a characterization of bipartite graphs in $\cFa$. In particular, Propositions~\ref{prop:PropertyH}, \ref{prop:PropertyT} and~\ref{prp:sufficient} lead us to the following question.

\begin{prob} \label{prob:2} Is there a condition stronger than Property T but weaker than Property $\rm{T}^*$ such that the connected bipartite graphs in $\cFa$ would be characterized by this condition and Property H?
\end{prob}

Finally, a different kind of condition (called Property U) was established, which characterizes all graphs in $\cFa$. The condition relies on certain connections between legal sequences and $\alpha$-sets in $G$, and as a result we could determine which of the graphs are in $\cFa$ within the class of $n$-crossed prisms. It will be interesting if one can use Property U to determine the graphs in $\cFa$ within some other natural class of graphs. 


\section*{Acknowledgement}

We are grateful to one of the reviewers for providing the beautiful proof of Proposition 12.
G.B. was supported in part by the National Research, Development and Innovation Office - NKFIH under the grant SNN 129364. B.B. was supported by the Slovenian Research Agency (ARRS) under the grants P1-0297, J1-2452, J1-3002, and J1-4008.

\end{document}